\documentclass[12pt, reqno]{amsart}
\pdfoutput=1

\usepackage{amsmath}
\usepackage{amsfonts}
\usepackage{amssymb}
\usepackage{amsthm}
\usepackage{breqn}
\usepackage{setspace}
\usepackage{fullpage}
\usepackage{enumitem}
\usepackage{bbold} 
\usepackage{comment}
\usepackage{hyperref}

\newtheorem{theorem}{Theorem} 
 
\newtheorem{definition}{Definition}
\newtheorem{claim}{Claim}

\newcommand{\B}{\mathcal{B}}

\newcommand{\h}{\mathcal{H}}

\newcommand{\abs}[1]{\left\lvert{#1}\right\rvert}

\DeclareMathOperator{\ex}{ex}

\title{On $3$-uniform hypergraphs avoiding \\ a cycle of length four}

\author{Beka Ergemlidze}
\address{Beka Ergemlidze, Department of Mathematics and Statistics, University of South Florida,
Tampa, Florida 33620, USA.}
\email{\texttt{beka.ergemlidze@gmail.com}}

\author{Ervin Gy\H{o}ri}
\address{Ervin Gy\H{o}ri, Alfr\'ed R\'enyi Institute of Mathematics, Budapest.}
\email{\texttt{gyori.ervin@renyi.mta.hu}} 

\author{Abhishek Methuku}
\address{Abhishek Methuku,  School of Mathematics,
 University of Birmingham,
Edgbaston,
Birmingham B15 2TT,
United Kingdom.}
\email{\texttt{abhishekmethuku@gmail.com}}

\author{Nika Salia}
\address{Nika Salia, Alfr\'ed R\'enyi Institute of Mathematics, Budapest.}
\email{\texttt{nikasalia@gmail.com}}

\author{Casey Tompkins}
\address{Casey Tompkins, Institute for Basic Science, Daejeon.}
\email{\texttt{ctompkins496@gmail.com}}

\thanks{The research of all authors was partially supported by the National Research, Development and Innovation Office NKFIH, grants K116769, K132696. The research of A.~Methuku was supported by the EPSRC, grant no. EP/S00100X/1 (A.~Methuku).
The research of N.~Salia was partially supported by the Shota Rustaveli National Science Foundation of Georgia SRNSFG, grant number FR-18-2499. The research of C.~Tompkins was supported by the Institute for Basic Science, IBS-R029-C1.}

\begin{document}
\maketitle

\begin{abstract}
In this note we show that the maximum number of edges in a $3$-uniform hypergraph without a Berge cycle of length four is at most $(1+o(1))\frac{n^{3/2}}{\sqrt{10}}$.
This improves earlier estimates by Gy\H{o}ri and Lemons and by F\"uredi and \"Ozkahya. 
\end{abstract}

\section{Introduction}

Given a hypergraph $\h$, let $V(\h)$ and $E(\h)$ denote the set of vertices and edges of $\h$. A hypergraph is called $r$-uniform if all of its edges have size $r$. Berge introduced the following definitions of a path and cycle in a hypergraph.


\begin{definition}
  A \emph{Berge cycle} of length $\ell$ in a hypergraph is a set of $\ell$ distinct vertices $\{v_1, \ldots, v_{\ell}\}$ and $\ell$ distinct edges $\{e_1, \ldots, e_{\ell}\}$ such that $\{v_i, v_{i+1}\} \subseteq e_i$ with indices taken modulo $\ell$. A \emph{Berge path} of length $\ell$ is a set of $\ell+1$ distinct vertices $\{v_1,\dots,v_{\ell+1}\}$ and $\ell$ distinct edges $\{e_1,\dots,e_\ell\}$ such that for $1 \le i \le \ell$ we have $\{v_i,v_{i+1}\} \subseteq e_i$.  
\end{definition}

Let $\ex_r(n,BC_{\ell})$ denote the maximum number of edges in a $r$-uniform hypergraph without a Berge cycle of length $\ell$. In the case $r=2$ we write simply $\ex(n,C_{\ell})$.

A well-known result of Bondy and Simonovits~\cite{bondy} asserts that for all $\ell \ge 2$ we have $\ex(n,C_{2\ell}) = O(n^{1+1/\ell})$, however the order of magnitude is only known to be sharp in the cases $\ell=2,3,5$.  Erd\H{o}s, R\'enyi and S\'os~\cite{ErdRenyiSos} proved the asymptotic result $\ex(n,C_4) = \frac{n^{3/2}}{2}+o(n^{3/2})$.  For hypergraphs of higher uniformity Gy\H{o}ri and Lemons~\cite{gyori} extended the Bondy Simonovits theorem and showed in particular that $\ex_r(n,BC_4) = O(n^{3/2})$. It follows from the results of F\"uredi and \"Ozkahya~\cite{furedi} that $\ex_3(n,BC_4) \le (1+o(1))\frac{2}{3}n^{3/2}$ (see Theorem~2 in~\cite{furedi}). In this note we significantly improve this bound as follows.


\begin{theorem}
\label{maintheorem}
\begin{displaymath}
(1-o(1))\frac{n^{3/2}}{3\sqrt{3}} \le \ex_3(n,C_4) \le (1+o(1))\frac{n^{3/2}}{\sqrt{10}}.
\end{displaymath}
\end{theorem}

 The lower bound comes from a construction originating in Bollob\'as and Gy\H{o}ri~\cite{bb} (stated more generally in~\cite{gyori0}). We take a $C_4$-free bipartite graph with color classes of size $n/3$ and $\frac{(2n/3)^{3/2}}{2\sqrt{2}} = \frac{n^{3/2}}{3\sqrt{3}}$ edges asymptotically.  For every vertex $v$ in one of the color classes, we take an additional vertex $v'$ and add it to every edge in the graph incident to $v$.  This results in a $3$-uniform hypergraph on $n$ vertices with $\frac{n^{3/2}}{3\sqrt{3}}$ edges asymptotically, and it is easy to verify this hypergraph contains no Berge~$C_4$.

\section{Proof of the upper bound in Theorem \ref{maintheorem}}

Now, we prove the upper bound. Let $\h$ be a $3$-uniform hypergraph with no Berge $C_4$ and no isolated vertices.  A block $\B$ of a hypergraph $\h$ is defined to be a maximal subhypergraph of $\h$  with the property that for any two edges $e,f \in E(\B)$, there is a sequence of edges of $\h$, $e=e_1,e_2,\dots,e_t=f$, such that $\abs{e_i \cap e_{i+1}} = 2$ for all $1 \le i \le t-1$ and $V(\B)=\cup_{h \in E(\B)}h$.  It is easy to see that the blocks of $\h$ define a unique partition of  $E(\h)$.

For a block $\B$ and an edge $h\in E(\B)$, we say $h$ is a \emph{leaf} if there exists $x\in h$ such that the only edge of $\B$ incident to $x$ is $h$. It is simple to observe that the set of non-leaf edges of a block $\B$ is either the empty set, a single edge or the edges of a complete hypergraph on $4$-vertices minus an edge, $K_4^{(3)-}$. Even more if the set of non-leaf edges of $\B$ is $E(K_4^{(3)-})$, then $\B=K_4^{(3)-}$.  This implies that the set $B(\h)=\{\B\mid \B$ is a block in $\h\}$ of all blocks of $\h$, can be partitioned into the following types of blocks: 

\begin{enumerate}
    
     \item  We say $\B\in B(\h)$ is \emph{type~1} if there exists an edge $e\in E(\B)$ such that  for all distinct $f_1,f_2\in E(\B)$, $f_1, f_2\neq e$, we have $\abs{e \cap f_i}=2$, for $i=1,2$ and $f_1\cap f_2 \subseteq e$.
    
    \item We say $\B\in B(\h)$ is \emph{type~2} if $\B=K_4^{(3)-}$.
    
\end{enumerate}


Define the $2$-shadow of a hypergraph to be the graph on the same set of vertices whose edges are all pairs of vertices $\{x,y\}$ for which there exists an edge $e \in E(\h)$ such that $\{x,y\} \subset e$.
We denote the $2$-shadow of a hypergraph $\h$ by $\partial \h$. The proof of Theorem~\ref{maintheorem} will proceed by estimating the number of $3$-paths ($3$-vertex paths) in the $2$-shadow of a Berge~$C_4$-free hypergraph in two different ways.  To this end, we introduce several notions of the degree of a vertex.  Given a vertex $v$ in a hypergraph $\h$, $d(v)$ denotes the classical hypergraph degree of $v$, in particular $d(v)=\abs{\{h\in E(\h): v\in h\}}$.  Let $d_s(v)$ be the (graph) degree of $v$ in the $2$-shadow of the hypergraph, in particular $d_s(v)=\abs{\{e\in E(\partial \h): v\in e\}}$.  Then, we define the \emph{excess degree} of the vertex $v$ to be $d_{ex}(v) = d_s(v) - d(v)$.  Finally, we define the \emph{block degree} $d_b(v)$ to be the total number of blocks containing an edge which contains $v$.

Notice that for every $4$-cycle $x_1,x_2,x_3,x_4,x_1$ of $\partial \h$, there exists three distinct integers $1\le i < j < k \leq 4$ such that $\{x_i,x_j,x_k\} \in E(\h)$, otherwise $\h$ contains a copy of Berge~$C_4$. We call this edge a \emph{representative edge} of this $4$-cycle. Note that  each $4$-cycle of $\partial \h$ has either $1$, $2$ or $3$  representative edges.  Two edges of $\h$ sharing two vertices yield a $C_4$ in $\partial \h$. However these are not only types of $C_4$'s in  $\partial \h$. We call a $4$-cycle  of $\partial \h$ \emph{rare} if the induced subhypergraph of $\h$ on the vertices of cycle does not contain  two edges sharing a diagonal pair of vertices of the $4$-cycle. In the following claim, we show that the number of such cycles is small.  


We define a particular type of $3$-path of $\partial \h$.
A $3$-path, $x_1,x_2,x_3$, is called \emph{good} if  $\{x_1,x_2,x_3\} \notin E(\h)$ and there is no $x\in V(\h)$ such that $x,x_1,x_2,x_3,x$ is a rare cycle of $\partial \h$.

\begin{claim}
\label{lemma:two_three_path}
For any $a,b \in V(\h)$, here are at most two good $3$-paths in $\partial \h$ with end points $a$ and $b$.
\end{claim}
\begin{proof}
 Suppose, by contradiction, that there are three distinct vertices  $v_1, v_2, v_3$ different from $a$ and $b$ such that $a,v_i,b$ forms a good $3$-path of $\partial \h$ for all integer $1\leq i\leq 3$.  It follows that there are three Berge paths $a,e_i,v_i,f_i,b$,  for all integer $1\leq i\leq 3$ in $\h$. Note that those edges are not necessarily distinct. But we have $e_i \neq f_i$ and  $e_i \neq f_j$, $i\neq j$, since  $\{a,v_i\} \subset e_i$ and $\{b,v_j\} \subset f_j$ and $\h$ is $3$-uniform. Note that if $e_2=e_3$, then $e_2=\{a,v_2,v_3\}$, hence $e_1\neq e_2$. Similarly we have either  $f_1\neq f_2$ or $f_1 \neq f_3$.  We may assume, without loss of generality, that $e_1\neq e_2, e_3$.  It follows that either  $a, e_1 , v_1, f_1 ,b, f_2 ,v_2, e_2, a$  or  $a, e_1 , v_1, f_1 ,b, f_3 ,v_3, e_3, a$ is a Berge $C_4$, a contradiction.
\end{proof}

\begin{claim}
\label{claim:rarec4}
There are at most $6\abs{E(\h)}$ rare $4$-cycles in $\partial \h$.
\end{claim}
\begin{proof}

We fix an edge $\{a,b,c\}\in E(\h)$. It suffices to show that the edge $\{a,b,c\}$ is representative of at most $6$ rare $4$-cycles (that is, $\{a,b,c\}$ is contained in the vertex set of at most $6$ rare $4$-cycles). Suppose by contradiction that this is not true. Observe that there are three possible positions for a fixed vertex $v$ among the vertices of a four cycle in $\partial \h$ with $\{a,b,c\}$. By the pigeonhole principle there are  $3$ distinct vertices $v_1,v_2,v_3$ different from $a$, $b$ or $c$ with the same position in the $4$-cycle. Without loss of generality, we may assume they form a $4$-cycle in the order $v_i,a,c, b,v_i$. Therefore from the definition of a rare $4$-cycle, there are at least three good $3$-paths in  $\partial \h$  from $a$ to $b$, a contradiction to Claim~\ref{lemma:two_three_path}.
\end{proof}

Using Claim~\ref{claim:rarec4}, it is easy to see that the number of $3$-paths in $\partial \h$ which are not good is at most $3\abs{E(\h)}+3\cdot 6\abs{E(\h)}=21\abs{E(\h)}$. Here we use the fact that each rare $4$-cycle induces an edge of $\h$.

By conditioning on the middle vertex of the $3$-path, we have the following estimate on the number of $3$-paths in $\partial \h$:
\begin{displaymath}
\#(\mbox{$3$-paths in $\partial \h$}) = \sum_{v \in V(\h)} \binom{d_s(v)}{2} = \sum_{v \in V(\h)} \binom{d(v)+d_{ex}(v)}{2}.
\end{displaymath}
 The following claim provides an upper bound on the number of good $3$-paths in $\partial \h$.
\begin{claim}\label{333}
\begin{displaymath}
\label{cases}
\#(\mbox{\emph{good $3$-paths in $\partial \h$}})  \leq 2\binom{n}{2}-4\sum_{v \in V(\h)}\binom{d_b(v)}{2}.
\end{displaymath}
\end{claim}

\begin{proof}
  Fix a vertex $v$ and consider two adjacent edges $\{v,x_1,x_2\}$ and $\{v,y_1,y_2\}$ such that they belong to the different blocks; clearly the vertices $v,x_1,x_2,y_1,y_2$ are all distinct. We claim that there is at most one good $3$-path, namely $x_i,v,y_j$, between $x_i$ and $y_j$, for each $i,j\in \{1,2\}$. Suppose this is not the case, then without loss of generality, there exists $u\neq v$ such that $x_1,u,y_1$ is a good $3$-path. By the definition of a good $3$-path, there are two distinct edges $h_x,h_y\in \h$ such that $x_1,u\in h_x$ and $y_1,u\in h_y$. If $\{v,x_1,x_2\}$, $\{v,y_1,y_2\}$, $h_x$ and $h_y$ are all different edges, then clearly there is a Berge $4$-cycle. Therefore either $\{v,x_1,x_2\}=h_x$ or $\{v,y_1,y_2\}=h_y$. Hence we have $u\in\{x_2,y_2\}$, without loss of generality we may assume $u=x_2$. Observe that the $4$-cycle $x_1,x_2,y_1,v$ of $\partial \h$  contains a  good $3$-path and so by definition the $4$-cycle $x_1,x_2,y_1,v$ is not a rare $4$-cycle. Hence we have a contradiction to the statement that edges $\{v,x_1,x_2\}$ and $\{v,y_1,y_2\}$  belong to the different blocks. Concluding that there is at most one good path between $x_i$ and $y_j$. So there are at least $4\sum_{v \in V(\h)}\binom{d_b(v)}{2}$ pairs of vertices which have at most one good $3$-path between them. From Claim~\ref{lemma:two_three_path}, for each pair of vertices there are at most two of  good $3$-paths in $\partial \h$. These observations complete the proof of Claim~\ref{333}.
  \end{proof}

Thus, since the number of $3$-paths which are not good is at most $21\abs{E(\h)}$, we have
\begin{equation}
\label{mainineq}
\sum_{v \in V(\h)} \binom{d(v)+d_{ex}(v)}{2} =\#(\mbox{$3$-paths in $\partial \h$}) \le 2\binom{n}{2}-4\sum_{v \in V(\h)}\binom{d_b(v)}{2} + 21\abs{E(\h)}.
\end{equation}
Now, we will obtain estimates for $\sum_{v \in V(\h)}d_{ex}(v)$ and $\sum_{v \in V(\h)}d_b(v)$. For each block $\B$ and $v\in V(\B)$, let $d_{ex}^{\B}(v)$ denote an excess degree of $v$ inside the hypergraph $\B$. If $\B$ is type~1, then every vertex  $v \in V(\B)$ has $d_{ex}^{\B}(v) \ge 1$, so for type~1 blocks, $\sum_{v \in V(\B)} d_{ex}^{\B}(v) \ge \abs{V(\B)}$. It is easy to see that for every block $\B$ we have $\abs{V(\B)} > \abs{E(\B)}$, so $\sum_{v \in V(\B)} d_{ex}^{\B}(v) > \abs{E(\B)}$, for every type~1 block $\B$.
If $\B$ is a type~2 block, then 
$\sum_{v \in V(\B)} d_{ex}^{\B}(v) =3=\abs{E(\B)}.$
Therefore, 
\[ \sum_{v \in V(\B)} d_{ex}^{\B}(v) \ge \abs{E(\B)}\]
for every block $\B$ in $B(\h)$. This together with the fact that the blocks define a partition of the edges $E(\h)$ implies

\begin{equation}
\label{bbound}
\sum_{v \in V(\h)}d_{ex}(v) =\sum_{\B \in B(\h)} \sum_{v \in V(\B)} d_{ex}^{\B}(v)\ge\sum_{\B \in B(\h)}\abs{E(\B)}=\abs{E(\h)}.
\end{equation}

On the other hand, a simple double counting argument yields 
\begin{displaymath}
\sum_{v \in V(\h)} d_b(v)=\sum_{\B \in B(\h)} \abs{V(\B)}.
\end{displaymath}

Therefore,

\begin{equation}
\label{dbbound}
\sum_{v \in V(\h)}d_{b}(v)=\sum_{\B \in B(\h)}\abs{V(\B) }\ge \sum_{\B \in B(\h)}\abs{\B}=\abs{E(\h)}.
\end{equation}

Now we will use the inequalities derived so far to get desired upper bound on $\abs{E(\h)}$.

By \eqref{bbound},
$$4\abs{E(\h)}=3\abs{E(\h)} +\abs{E(\h)}\le\sum_{v \in V(\h)}(d(v)+d_{ex}(v)).$$

Since $\binom x2$ is a convex function, by Jensen's inequality we have
\[\binom {\frac 1n\sum_{v \in V(\h)}(d(v)+d_{ex}(v))}2\le \frac 1n \sum_{v \in V(\h)} \binom{d(v) +d_{ex}(v)}{2}.\]

Combining the above two inequalities we get
\begin{equation}
\label{mainest1}
n \binom{\frac{4\abs{E(\h)}}{n}}{2}
\le \sum_{v \in V(\h)} \binom{d(v) +d_{ex}(v)}{2}.
\end{equation}
Similarly, by~\eqref{dbbound} and Jensen's inequality, we have
\begin{equation}
\label{mainest2}
n \binom{\frac{\abs{E(\h)}}{n}}{2} \le \sum_{v \in V(\h)} \binom{d_b(v)}{2}.
\end{equation}
Combining \eqref{mainineq}, \eqref{mainest1} and \eqref{mainest2} we obtain
\begin{equation}
\label{final}
n \binom{\frac{4\abs{E(\h)}}{n}}{2} +4n \binom{\frac{\abs{E(\h)}}{n}}{2} \le 2 \binom{n}{2} + 21 \abs{E(\h)}.
\end{equation}
Rearranging \eqref{final} yields the desired bound,
\begin{displaymath}
\abs{E(\h)} \le (1+o(1))\frac{n^{3/2}}{\sqrt{10}}. \qedhere
\end{displaymath}

\end{document}